\documentclass[a4paper,10pt, oneside, reqno]{amsart}
\pdfoutput=1

\usepackage[activate={true,nocompatibility},final,tracking=false,kerning=true,spacing=true,factor=1100,stretch=0,shrink=0]{microtype}

\usepackage[T1]{fontenc} 
\usepackage[utf8]{inputenc}
\usepackage[british]{babel}
\usepackage{mathrsfs}
\usepackage{amsfonts}
\usepackage{amsmath} 
\usepackage{amsthm}  
\usepackage{amssymb}
\usepackage{mathtools}
\mathtoolsset{showonlyrefs}
\usepackage{graphicx}
\usepackage{xcolor}
\usepackage{url}
\usepackage{cite}
\usepackage{xifthen}
\usepackage{stmaryrd}
\usepackage{centernot}
\usepackage{indentfirst}
\usepackage[hang,flushmargin]{footmisc}
\usepackage{geometry}
\usepackage{wasysym}
\usepackage{hyperref}
\hypersetup{colorlinks=false, pdfborder={0 0 0}}
\usepackage{cleveref}
\usepackage{tikz}
\usetikzlibrary{matrix,arrows,decorations.pathmorphing, decorations.markings, patterns}
\newcommand{\cat}{%
  \mathbin{\mathpalette\dotieconcat\relax}%
}
\newcommand{\dotieconcat}[2]{
  \text{\raisebox{.8ex}{$\smallfrown$}}%
}
\usepackage[inline]{enumitem}
\setlist[1]{labelindent=\parindent}
\setlist[enumerate,1]{label = \arabic*.,ref   = \arabic*}
\setlist[enumerate,2]{label = \alph*),ref   = \theenumi\alph*)}
\setlist[enumerate,3]{label = \roman*),ref   = \theenumii.\roman*}

\newcommand{\from}{\colon}

\newcommand{\implica}{\rightarrow}
\newcommand{\coimplica}{\leftrightarrow}
\renewcommand{\phi}{\varphi}
\renewcommand{\epsilon}{\varepsilon}
\renewcommand{\models}{\vDash}

\newcommand{\monster}{\mathfrak U}

\newcommand{\sse}{\Leftrightarrow}
\newcommand{\sop}[1]{\ensuremath{#1\mathsf{-SOP}}}
\newcommand{\sigof}[1]{\ensuremath{\Sigma_{\mathsf{\MakeUppercase{#1}}}}}
\newcommand{\sopp}[1]{\sop{\sigof{#1}}}
\DeclareMathOperator{\age}{Age}

\DeclarePairedDelimiter{\set}{\{}{\}}
\DeclarePairedDelimiter{\abs}{\lvert}{\rvert}
\DeclarePairedDelimiter{\seq}{\langle}{\rangle}
\theoremstyle{definition}
\newtheorem{defin}{Definition}[section]
\newtheorem{thm}[defin]{Theorem}
\newtheorem{pr}[defin]{Proposition}
\newtheorem{eg}[defin]{Example}
\newtheorem{rem}[defin]{Remark}
\newtheorem{question}[defin]{Question}

\theoremstyle{theorem}
\newtheorem*{unmdthm}{Theorem}

\title{Model-theoretic dividing lines via posets}

\author{Dar\'io Garc\'ia}
\address{Dar\'io Garc\'ia, Departamento de Matemáticas. Universidad de los Andes, Carrera 1 No. 18A-10, Edificio H, Bogot\'a 111711, Colombia.}
\email{da.garcia268@uniandes.edu.co}
\author{Rosario Mennuni}
\address{Rosario Mennuni, Dipartimento di Matematica, Universit\`a di Pisa, Largo Bruno Pontecorvo 5, 56127 Pisa, Italy}
\email{R.Mennuni@posteo.net}
\date{\today}



\begin{document}
\maketitle
\begin{abstract}
We show that for each property $\mathsf{P}\in \set{\mathsf{OP}, \mathsf{IP}, \mathsf{TP}_1, \mathsf{TP}_2, \mathsf{ATP}, \mathsf{SOP}_3}$ there is a poset $\sigof{\mathsf{P}}$ such that a theory has property $\mathsf{P}$ if and only if some model interprets a poset in which $\sigof{\mathsf{P}}$ can be embedded. We also introduce a new property $\mathsf{SUP}$,  consistent with $\mathsf{NIP}_2$ and implying $\mathsf{ATP}$ and $\mathsf{SOP}$.
\end{abstract}
\section*{Introduction}
One of the central themes of contemporary model theory is the study of \emph{dividing lines},  which are properties used to assess the combinatorial and geometrical ``tameness'' of first-order theories. Several of these are obtained by declaring that a theory is ``wild'' if a certain pattern may be found in a family of uniformly definable sets, and ``tame'' otherwise.

The archetypical dividing line is the \emph{order property} $\mathsf{OP}$. A theory is \emph{stable} iff no formula defines a linear order on an infinite (not necessarily definable) set, which is equivalent to \emph{not} having the order property, or being $\mathsf{NOP}$. Stability theory has been the topic of several monographs~\cite{baldwin_fundamentals_1988, buechler_essential_2017, pillay_geometric_1996} and is featured in many courses in model theory~\cite{poizat_course_2000, tent_course_2012}.
In addition to being crucial in Shelah's classification program~\cite{shelah_classification},  stability played a key role in applications of model theory to algebra~\cite{bouscaren}.

In the last 30 years other notable dividing lines, in large part defined by Shelah ---~a recent exception is the \emph{antichain tree property} $\mathsf{ATP}$, defined by Ahn and Kim~\cite{ahnkimatp}~---  have been studied in detail and have produced important applications in geometry and combinatorics.
For instance, $\mathsf{OP}$ is equivalent~\cite{shelah_classification} to the disjunction of the \emph{independence property} $\mathsf{IP}$  with the \emph{tree property} $\mathsf{TP}$. The (lack of the) independence property (see~\cite{simon_guide_2015}) was used, among other things, in the proof of the celebrated \emph{Pillay conjectures} on o-minimal theories~\cite{hrushovski_nip_2011}.  The negation of the tree property corresponds to \emph{simplicity} of a theory \cite{casanovas_simple_2011, kim_simplicity_2014, wagner_simple_2000}, and $\mathsf{TP}$ is itself equivalent \cite{shelah_simple_1980} to the disjunction of the tree properties of the first and second kind $\mathsf{TP}_1$ and $\mathsf{TP}_2$ (see e.g.~\cite{ahnkimatp} for their definitions).

Such a plethora of properties (whose study is known as \emph{neostability theory}) calls for a uniform treatment, and while one was proposed in \cite{shelah_what_1999}, to be best of the authors' knowledge it never picked up considerable steam. Here, we show that several dividing lines may be defined parametrically, where the parameter in question may be taken to be an infinite poset, or a hereditary class of finite posets.

Recall that the \emph{strict order property} is defined by declaring  a theory to be $\mathsf{SOP}$ iff it has a model interpreting a poset with an infinite chain, that is, a poset in which $\omega$ may be embedded. We define a family of properties $\sop{\Sigma}$ by replacing the $\omega$ in the above definition with an arbitrary poset $\Sigma$,  and show that choosing $\Sigma$ carefully allows to define in this fashion several known dividing lines.

An immediate consequence of the definitions, formally given in \Cref{sec:defs}, is the introduction of a new property $\mathsf{SUP}$,  which implies being on the ``wild'' side of all known dividing lines, with the exception of those higher in the $\mathsf{IP}_k$ hierarchy~\cite{onndep}. Namely, $\mathsf{SUP}$ implies $\mathsf{SOP}$ (hence $\mathsf{TP}_1$) and $\mathsf{ATP}$ (hence $\mathsf{TP}_2$, hence $\mathsf{IP}$), and there is a theory which is $\mathsf{SUP}$ and $\mathsf{NIP}_2$.

The rest of the paper is mostly devoted to proving the following theorem.
\begin{unmdthm}
  For each property $\mathsf{P}\in \set{\mathsf{OP}, \mathsf{IP}, \mathsf{SOP}, \mathsf{SUP}, \mathsf{TP}_1, \mathsf{TP}_2, \mathsf{ATP}, \mathsf{SOP}_3}$ there is a poset $\sigof{\mathsf{P}}$ such that a theory has $\mathsf{P}$ if and only if some model interprets a poset in which $\sigof{\mathsf{P}}$ can be embedded.
\end{unmdthm}
The $\mathsf{OP}$ and $\mathsf{IP}$ parts are proven in \Cref{sec:opip}. The last four properties are treated in \Cref{sec:tps}. We do this uniformly by introducing the notion of a \emph{maximal consistency pattern}, building posets corresponding to properties that may defined by using such a pattern, and observing that  $\mathsf{TP}_1$, $\mathsf{TP}_2$, $\mathsf{ATP}$, and $\mathsf{SOP}_3$ all fall into this class. While, by~\cite{mutchnik}, $\mathsf{SOP}_1$ coincides with $\mathsf{TP}_1$, hence fits in our framework, it is still open whether $\mathsf{TP}$ may be treated in this fashion. This is the content of one of a handful of open questions featured in our concluding \Cref{sec:future}.

\section{Definitions}\label{sec:defs}
Fix a complete first-order $T$ and work in a monster model $\monster$. Variables and parameters are finite tuples unless otherwise stated. The length of a tuple $x$  is denoted by $\abs x$. We write e.g.\ $\models\phi(x)\implica \psi(x)$ with the meaning $\monster\models \forall x\; \phi(x)\implica \psi(x)$.
\begin{defin}\*\label{defin:sigmasop}
  \begin{enumerate}
  \item Let $\mathcal K$ be a class of finite posets. A partitioned formula
    $\psi(x;y)$ has $\sop {\mathcal K}$ iff for every $(\sigma,\le)\in {\mathcal K}$ there
    are $(d^s\mid s\in \sigma)$ in $\monster^{\abs y}$ such that $\models\psi(x; d^{s_0})\implica \psi(x; d^{s_1}) \sse \sigma\models s_0\le s_1$.
\item We say that $T$ has $\sop {\mathcal K}$ iff some partitioned formula has $\sop {\mathcal K}$.
\item If $(\Sigma, \le)$ is an arbitrary poset, we write $\sop \Sigma$ instead of $\sop{\age(\Sigma)}$.
\end{enumerate}
\end{defin}
\begin{rem}\*\label{rem:otherproperties}
  \begin{enumerate}
    \item The usual   $\mathsf{SOP}$ is the same as $\sop {\mathcal K}$ for ${\mathcal K}$ the class of finite linear orders, or as $\sop \Sigma$ for some (equivalently every) infinite linear order $\Sigma$.
  \item If $\mathcal K$ is the class of all finite posets then
    clearly having $\sop{\mathcal K}$ implies having all possible
    $\sop{\mathcal K'}$. We call this property the \emph{strict universal-order property}  $\mathsf{SUP}$. 
\item The class $\mathcal K$ of all finite posets is known to be Fra\"iss\'e. Its Fra\"iss\'e limit, the so-called generic poset, eliminates quantifiers in the language $\set{\le}$, hence it is $\mathsf{NIP}_2$ by~\cite[Proposition~6.5]{onndep}. Clearly, the generic poset has $\mathsf{SUP}$, hence no property of the form $\sop{\mathcal K}$ can imply $\mathsf{IP}_2$.
\end{enumerate}
\end{rem}

\begin{rem}
  All the properties under consideration are preserved by naming or forgetting constants. Therefore, we may freely assume to have enough $\emptyset$-definable elements, say $0,1,2$, etc. These elements will be used simply to perform case distinctions, by adding conjuncts such as $t=0$. If the reader prefers to work over $\emptyset$ instead, case distinctions can be done by increasing the number of parameter variables.
\end{rem}
The following proposition will be used repeatedly without mention. Recall that $f\from \Sigma\to \Sigma'$ is an embedding of posets iff $\Sigma\models s_0\le s_1\iff \Sigma'\models f(s_0)\le f(s_1)$.
\begin{pr} The following are equivalent for a first-order theory $T$.
  \begin{enumerate}
  \item  $T$ has $\sop \Sigma$.
  \item There are a partitioned formula     $\psi(x;y)$ and a $\Sigma$-sequence of tuples $(d^s\mid s\in \Sigma)$ in $\monster^{\abs y}$ such that    $\models\psi(x; d^{s_0})\implica \psi(x; d^{s_1}) \iff \Sigma\models s_0\le s_1$.
  \item $\monster$ interprets a poset $\Sigma'$ such that there is an embedding of posets $\Sigma\to \Sigma'$.
  \end{enumerate}
\end{pr}
\begin{proof}
The proof is  straightforward, using compactness and the definable equivalence relation $E(y_0, y_1)\coloneqq \forall x\; \phi(x; y_0)\coimplica \phi(x;y_1)$. We leave the details to the reader.
\end{proof}

\section{Order and independence}\label{sec:opip}
In this section we show that the order property and the independence property may be phrased as $\sop {\Sigma}$ for suitable posets $\Sigma$.
\begin{defin}\label{defin:sigofop}
  We define $\sigof{op}$ to be the poset with domain    $\set{\alpha_i\mid i<\omega}\cup \set{\beta_j\mid j<\omega}$
  whose only relations are given by $\alpha_i< \beta_j\iff  i<j$.
\end{defin}
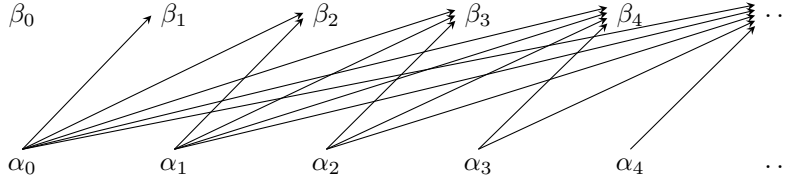
\begin{figure}[h]
  \begin{tikzpicture}[scale=1]
    \pgfmathsetmacro{\opnodnum}{6}
    \pgfmathsetmacro{\opnodnummu}{\opnodnum-1}        
    \foreach \i in{1,...,\opnodnummu}{\pgfmathparse{int(\i-1)}\pgfmathtruncatemacro{\imu}{\pgfmathresult}\node(a\i) at (2*\i,0) {$\alpha_{\imu}$};}
    \foreach \j in{1,...,\opnodnummu}{\pgfmathparse{int(\j-1)}\pgfmathtruncatemacro{\jmu}{\pgfmathresult}\node(b\j) at (2*\j,2) {$\beta_{\jmu}$};}
    \node (a\opnodnum) at (2*\opnodnum, 0) {\dots};
    \node (b\opnodnum) at (2*\opnodnum, 2) {\dots};
        \foreach \i in {1,...,\opnodnummu}{\pgfmathparse{int(\i+1)}\pgfmathsetmacro{\ipu}{\pgfmathresult}
          \foreach \j in{\ipu,...,\opnodnum}{
          \pgfmathparse{int(-\i*2+\j)}\pgfmathsetmacro{\arrsp}{\pgfmathresult}
          \draw[->, >=stealth] (a\i.north)--(
          [yshift=\arrsp pt] 
          b\j.west);}
          }
      \end{tikzpicture}\caption{A portion of the poset $\sigof{op}$ from \Cref{defin:sigofop}.  Arrows denote strict inequalities.}\label{figure:sigofop}
\end{figure}

 This poset is partially depicted in \Cref{figure:sigofop}. It has height $2$, with all the $\alpha_i$ and $\beta_0$ having height $0$, and all $\beta_j$ for $j>0$ having height $1$. Throughout this paper, we will refer to the set of elements of height $n$ in  $\Sigma$ as the \emph{$n$-th level} of $\Sigma$. 
\begin{thm}\label{thm:soppop}
  $T$ has \sopp{op} if and only if $T$ has $\mathsf{OP}$.
\end{thm}
\begin{proof}
Left to right, suppose that $\psi(x; w)$ has $\sopp{op}$. Hence, there are  $(d^{\alpha_i})_{i\in \omega}$, and $(d^{\beta_j})_{j\in \omega}$ such that
\[
\models   \psi(x; d^{\alpha_i})\implica \psi(x; d^{\beta_j})\iff i<j
\]
From the above, it is immediate to observe that the partitioned formula  $\phi(y;z)\coloneqq \forall x\; \psi(x; y)\implica \psi(x; z)$ has $\mathsf{OP}$.

Right to left, let $\phi(x;y)$ have  $\mathsf{OP}$.  By padding, we may assume $\abs x=\abs y$. By compactness, there are tuples  $(a^i)_{i\in \mathbb Z}$ and $(b^j)_{j\in \mathbb Z}$ such that $\models\phi(a^i; b^j)\iff i<j$. Let $0,1$ be two distinct elements of $\monster^{\abs x}$ and define  the partitioned formula
\[
    \psi(x; y_0, y_1,y_2, t)\coloneqq (t=0\land x=y_0)\lor(t=1\land ((\neg \phi(x; y_0)\land \phi(x; y_1))\lor \phi(x; y_2))
  \]
  For ease of notation, instead of working with $\sopp{op}$ directly, we work with the isomorphic subposet with domain $\set{\alpha_i\mid 0<i<\omega}\cup \set{\beta_j\mid 0<j<\omega}$.
For $i>0$, set $d^{\alpha_i}\coloneqq(a^i, 0,0,0)$,  and for $j>0$ set $d^{\beta_j}\coloneqq(b^0, b^j, b^{-j},1)$. Observe the following.
\begin{itemize}
\item For every $i>0$, the formula $\psi(x;d^{\alpha_i})$ defines the singleton $\set{a^i}$. In particular, if $i_0\ne i_1$, then no inclusion holds between $\psi(x;d^{\alpha_{i_0}})$ and $\psi(x;d^{\alpha_{i_1}})$.
\item For every $j>0$ and $i\in \mathbb Z$, the formula $\psi(a^i;d^{\beta_j})$ holds if and only if $i\in (-\infty, -j)\cup [0,j)$. In particular,  if $j_0< j_1$, then $a^{j_0}\models\psi(x; d^{\beta_{j_1}})\land \neg \psi(x; d^{\beta_{j_0}})$, while $a^{-j_1}\models\psi(x; d^{\beta_{j_0}})\land \neg \psi(x; d^{\beta_{j_1}})$. It follows that, if $j_0\ne j_1$, then  no inclusion holds between $\psi(x;d^{\beta_{j_0}})$ and $\psi(x;d^{\beta_{j_1}})$.
 \end{itemize}
 It follows easily from the above that, if $i,j>0$, then  $\models\psi(x;d^{\alpha_i})\implica \psi(x;d^{\beta_j})$ if and only if $i<j$. From this, we conclude that $\psi(x; y_0,y_1,y_2,t)$,  has $\sopp{op}$.
\end{proof}

\begin{defin}
  We define $\sigof{ip}$ to be the poset with domain
  \[
    \set{\alpha_i\mid i<\omega}\cup \set{\beta_W\mid W\subseteq \omega}
  \]
  whose only relations are given by $\alpha_i< \beta_W\iff  i\in W$.
\end{defin}
Observe this poset again has two levels, the $0$-th one given by the $\alpha_i$ and $\beta_\emptyset$, and the other one containing all the other $\beta_W$.
\begin{thm}
  $T$ has \sopp{ip} if and only if $T$ has $\mathsf{IP}$.
\end{thm}
\begin{proof}
  Left to right, suppose that $\psi(x; w)$ has $\sopp{ip}$. Hence, there are  $(d^{\alpha_i})_{i< \omega}$, and $(d^{\beta_W})_{W\subseteq \omega}$ such that, in particular, 
\[
\models  \psi(x; d^{\alpha_i})\implica \psi(x;d^{\beta_W} )\iff i\in W
\]
From this it is immediate that $\phi(y;z)\coloneqq \forall x\; \psi(x; y)\implica \psi(x; z)$ has $\mathsf{IP}$.

Right to left, let $\phi(x;y)$ have  $\mathsf{IP}$. By padding, we may assume $\abs x=\abs y$. Fix tuples  $(a^i)_{i\in \omega}$ and $(b^W)_{W\subseteq \omega}$ such that $\models\phi(a^i; b^W)\iff i\in W$. Let $0,1$ be two distinct elements of $\monster^{\abs x}$  and define the partitioned formula
\[
    \psi(x; y, t)\coloneqq (t=0\land x=y)\lor(t=1\land \phi(x; y))
  \]
For $i\in \omega$, set $d^{\alpha_i}\coloneqq(a^i, 0)$. Fix now an infinite almost disjoint family\footnote{Recall that an \emph{almost disjoint family} is a family $\mathcal W$ of infinite sets such that for every distinct $W_0, W_1\in \mathcal W$ the intersection $W_0\cap W_1$ is finite.} $\mathcal W$ of subsets of $\omega$ and, for $W\in \mathcal W$, set $d^{\beta_W}\coloneqq(b^W, 1)$.   Observe the following.
\begin{itemize}
\item For every $i\in \omega$, the formula $\psi(x;d^{\alpha_i})$ defines the singleton $\set{a^i}$. In particular, if $i_0\ne i_1$, then no inclusion holds between $\psi(x;d^{\alpha_{i_0}})$ and $\psi(x;d^{\alpha_{i_1}})$.
\item Because the family $\mathcal W$ is almost disjoint,  no inclusion holds between distinct elements of $\mathcal W$. It follows that, for $W_0\ne W_1\in \mathcal W$, no inclusion holds between $\psi(x; d^{\beta_{W_0}})$ and $\psi(x; d^{\beta_{W_1}})$, since if $i\in W_{k}\setminus W_{1-k}$, for $k\in \set{0,1}$, then $a^i\models \phi(x; d^{\beta_{W_k}})\land \neg \phi(x; d^{\beta_{W_{1-k}}})$.
\item Clearly, $\models \psi(x; d^{\alpha_i})\implica \psi(x; d^{\beta_W})\iff i\in W$.
\end{itemize}
To conclude that $\psi(x; y,t)$ has $\sigof{ip}$ we just need to show that, for every $n$, the bipartite membership digraph between $n$ and its powerset embeds in the membership digraph between $\omega$ and an infinite almost disjoint family $\mathcal W$.  To this end, take an almost disjoint family $\mathcal W_n$ of subsets of $\omega\setminus n$, and let $W_0,\ldots, W_{2^n-1}$ be pairwise distinct elements of $\mathcal W_n$. Now fix a bijection $f\from 2^n\to \mathscr P(n)$, and let $\mathcal W$ be obtained from $\mathcal W_n$ by replacing $W_i$ with $W_i\cup f(i)$.
\end{proof}

\section{Some tree properties}\label{sec:tps}
Here we deal with the cases of $\mathsf{TP}_1$, $\mathsf{TP}_2$, $\mathsf{ATP}$, and $\mathsf{SOP}_3$. We subsume these properties under the more general notion of a \emph{tree property} defined via a \emph{maximal consistency pattern}.
\begin{defin}\*
  \begin{enumerate}
  \item A \emph{consistency pattern} $\mathcal P$ is a triple $\mathcal P=(J, \mathcal I, \mathcal C)$ such that $J$ is a set, $\mathcal I\subseteq J^{[2]}$ a family of subsets of $J$ of size $2$ (called the \emph{inconsistent pairs}) and  $\mathcal C\subseteq \mathscr P(J)$ a family of subsets of $J$ (called the \emph{consistent sets}), satisfying the following.    \begin{enumerate}[label=(C\arabic*)]
      \item None of $J, \mathcal I, \mathcal C$ is empty.
      \item  There are no $I\in \mathcal I$ and $C\in \mathcal C$ such that  $I\subseteq C$ or $C\subseteq I$.
    \end{enumerate}
    \item A consistency pattern is \emph{maximal} iff the following  conditions are satisfied.
      \begin{enumerate}[label=(M\arabic*)]
      \item\label{point:pairsdecided} For every $\set{i,j}\in J^{[2]}\setminus \mathcal I$ there
        is $C\in \mathcal C$ such that $\set{i,j}\subseteq C$. In
        other words, every pair of distinct indices is either
        inconsistent or contained in a consistent set.

      \item\label{point:cmaximal}   For every $C\in \mathcal C$ and every $j\in J\setminus C$ there is $i\in C$ such that $\set{i,j}\in \mathcal I$. In other words, consistent sets are maximal.
      \item\label{point:allareinconsistentwithsome}  For every $j\in J$ there are at least two $i\in J$ such that $\set{i,j}\in \mathcal I$.
\end{enumerate}

  \item A \emph{tree property} $\mathsf{P}$ is defined by fixing a consistency pattern $\mathcal P=(J, \mathcal I, \mathcal C)$ and saying that a partitioned formula $\phi(x;y)$ has
    $\mathsf{P}$ iff there is a $J$-sequence of parameters $(b^j\mid j\in J)$ in $\monster^{\abs y}$ such that
\begin{enumerate}
\item for every $C\in \mathcal C$, the set of formulas $\set{\phi(x; b^j)\mid j\in C}$ is consistent, and
\item for every $I\in \mathcal I$, the set of formulas $\set{\phi(x; b^j)\mid j\in I}$ is inconsistent.
\end{enumerate}
  A theory has $\mathsf{P}$ iff some partitioned formula has $\mathsf{P}$.
\end{enumerate}
\end{defin}
\begin{rem}\label{rem:weakmax}
Condition~\ref{point:cmaximal} implies that for every $C\ne C'\in \mathcal C$ there are $i\in C$ and $j\in  C'$ such that $\set{i,j}\in \mathcal I$. In other words, any two consistent sets differ on a pair of inconsistent indices. 
\end{rem}
\begin{eg}The properties $\mathsf{TP}_1$, $\mathsf{TP}_2$, and $\mathsf{ATP}$ can all be defined via maximal consistency patterns, with index sets respectively $\omega^{<\omega}\setminus\set{\seq{}}$, $\omega^2$, and $2^{<\omega}\setminus\set{\seq{}}$ (for $\mathsf{ATP}$, note that it can be equivalently defined using \emph{maximal} antichains). Note that $\mathsf{TP}_1$ may also be defined as $\mathsf{SOP}_2$, with index set $2^{<\omega}\setminus\set{\seq{}}$, and $\mathsf{TP}_2$ may also be defined with index set $\omega^{<\omega}$ (this is folklore, the idea is to  make every row of the matrix into the set of siblings of a node).
\end{eg}
\begin{eg}\label{eg:tpnonmax}
   The usual tree property $\mathsf{TP}$ can be defined in this fashion by taking $J={\omega^{<\omega}}\setminus \set{\seq{}}$, 
    while as $\mathcal C$ we take the family of branches of $J$, and as
    $\mathcal I$ the family of pairs of siblings. In other words,
    $I\in J^{[2]}$ is in $\mathcal I$ if and only if there  $\sigma\in \omega^{<\omega}$ such that both elements of $I$ are of the form    $\sigma\cat i$ for some $i\in \omega$. The removal of the root ensures that \ref{point:allareinconsistentwithsome} is satisfied, but   the consistency pattern above does not satisfy \ref{point:pairsdecided} nor  \ref{point:cmaximal}, hence it is not maximal (although it still satisfies the property in \Cref{rem:weakmax}). 
\end{eg}
\begin{rem}\label{rem:sop1sop2}
  The same remarks made in \Cref{eg:tpnonmax} also apply to the usual definition of $\mathsf{SOP}_1$. Nevertheless, by~\cite{mutchnik}, $\mathsf{SOP}_1$ equals $\mathsf{SOP}_2$, equivalently $\mathsf{TP}_1$. Therefore, it may happen that the same tree property may be defined via a maximal consistency pattern and via a non-maximal one.
\end{rem}

\begin{eg}\label{eg:sop3}
The property  $\mathsf{SOP}_3$ can be defined through the following maximal consistency pattern, partially depicted in \Cref{figure:sop3pat}.   Set $J\coloneqq \mathbb Z\times\set{0,1}$, take $\mathcal I=\set{\set{(i,0), (j,1)}\mid i\ge j}$ and $\mathcal C\coloneqq \set{C_k\mid k\in \mathbb Z}$, where $C_k\coloneqq\set{(i,0)\mid i<k}\cup \set{(j,1)\mid j\ge k}$.  Maximality is easily checked. The fact that  $\mathsf{SOP}_3$ implies the tree property given by this pattern follows from \cite[Claim~2.19]{shelah_toward_1996}, compactness, and case distinctions such as replacing $\phi(x;y)$ and $\psi(x;y)$ by $\theta(x;yt)\coloneqq ((t=0\land\phi(x;y))\lor (t=1\land \psi(x;y)))$. The other direction relies on the fact that in point (b) in \cite[Claim~2.19]{shelah_toward_1996} the inconsistency of $\set{\phi(x;y), \psi(x;y)}$ is superfluous, see e.g.~\cite[Lemma~5.5]{ramseynotes}.
\end{eg}
\begin{figure}[h]
  \begin{tikzpicture}[scale=1]
    \pgfmathsetmacro{\opnodnum}{7}
    \pgfmathtruncatemacro{\opnodnummu}{\opnodnum-1}        
    \foreach \i in{1,...,\opnodnummu}{\pgfmathparse{int(\i-2)}\pgfmathtruncatemacro{\imu}{\pgfmathresult}\node(a\i) at (2*\i,0) {$(\imu,0)$};}
    \foreach \j in{1,...,\opnodnummu}{\pgfmathparse{int(\j-2)}\pgfmathtruncatemacro{\jmu}{\pgfmathresult}\node(b\j) at (2*\j,2) {$(\jmu,1)$};}
    \node (a0) at (0, 0) {\dots};
    \node (b0) at (0, 2) {\dots};
    \node (a\opnodnum) at (2*\opnodnum, 0) {\dots};
    \node (b\opnodnum) at (2*\opnodnum, 2) {\dots};

    \draw plot [smooth, tension=0.3] coordinates 
    {([yshift=5]a0.north west) (a2.north east) ([xshift=25]b2.north east) ([yshift=4]b\opnodnum.north  east)};
    \draw plot [smooth, tension=0.6] coordinates 
    {([yshift=-4]a0.south west) ([xshift=10, yshift=-3]a2.south east)  ([xshift=0, yshift=-2]b3.south)  ([yshift=-5]b\opnodnum.south east)};

    \draw [dashed] plot [smooth cycle, tension=1] coordinates 
    {(a3.east)([yshift=3]b2.east)(b2.west)([yshift=-3]a3.west)};
      \end{tikzpicture}\caption{A portion of the pattern defining $\mathsf{SOP}_3$ from \Cref{eg:sop3}. The solid lines enclose a consistent set, the dashed line encloses an inconsistent pair.}\label{figure:sop3pat}
\end{figure}
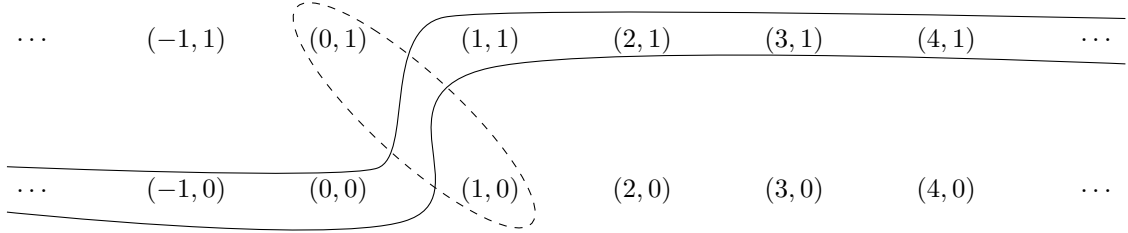
\begin{rem}\label{rem:negcons}
  If $T$ has a tree property $\mathsf{P}$, then this can be witnessed by a partitioned formula $\phi(x;y)$ and a $J$-sequence $(b^j\mid j\in J)$ in $\monster^{\abs y}$ such that  $\set{\neg\phi(x; b^j)\mid j\in J}$ is consistent and has infinitely many realisations: 
it is enough to artificially enlarge the sort of $x$, for instance by replacing $x$ with $xt$ and changing $\phi(x;y)$ to $\phi(x;y)\land t=0$, for a fixed element $0$.
\end{rem}

\begin{defin}\label{defin:S}
Let $\mathcal P=(J,\mathcal C,\mathcal I)$ be a maximal consistency pattern. Define a set $\sigof{\mathcal P}$ as follows.
  \begin{enumerate}
  \item For every $j\in J$, there are two points $\beta_j$ and $\gamma_j$.
  \item For every $C\in \mathcal C$, there are two points $\alpha_C$ and $\delta_C$.
  \item $\Sigma$ has no other point.
  \end{enumerate}
  Now, we define a binary relation $<$ on $\Sigma$ as follows.
  \begin{enumerate}[resume]
  \item \label{point:alb} Whenever $i\in C\in \mathcal C$, we have $\alpha_C<\beta_i$.
  \item \label{point:blc}Whenever $i\in C\in \mathcal C$, we have $\gamma_i<\delta_C$.
  \item \label{point:cld}Whenever $\set{i,j}\in \mathcal I$, we have $\beta_i< \gamma_j$ (and $\beta_j< \gamma_i$).
  \item Let $<$ be the transitive closure of the relations above.
  \end{enumerate}
\end{defin}
\begin{figure}[t]
  \begin{tikzpicture}[scale=1]
    \node (bi1) at (-4,2) {$\beta_{i_1}$};
    \node (ci1) at (-4,4) {$\gamma_{i_1}$};
    
    \node (ac) at (-3,1) {$\alpha_C$};
    \node (bi0) at (-2,2) {$\beta_{i_0}$};
    \node (ci0) at (-2,4) {$\gamma_{i_0}$};
    \node (dc) at (-3,5) {$\delta_C$};
    
    \node (ad) at (3,1) {$\alpha_D$};
    \node (bj0) at (2,2) {$\beta_{j_0}$};
    \node (cj0) at (2,4) {$\gamma_{j_0}$};
    \node (dd) at (3,5) {$\delta_D$};
    
    \node (bj1) at (4,2) {$\beta_{j_1}$};
    \node (cj1) at (4,4) {$\gamma_{j_1}$};

    \draw[->, >=stealth] (ac)--(bi0);
    \draw[->, >=stealth] (ac)--(bi1);
    \draw[->, >=stealth] (ci0)--(dc);
    \draw[->, >=stealth] (ci1)--(dc);

    \draw[->, >=stealth] (ad)--(bj0);
    \draw[->, >=stealth] (ad)--(bj1);
    \draw[->, >=stealth] (cj0)--(dd);
    \draw[->, >=stealth] (cj1)--(dd);

    \draw[->, >=stealth] (bi1)--(cj0.west);
    \draw[->, >=stealth] (bi1)--(cj1);
    \draw[->, >=stealth] (bi0)--(cj0);
    \draw[->, >=stealth] (bi0)--(cj1.south west);

    \draw[->, >=stealth] (bj1)--(ci0.east);
    \draw[->, >=stealth] (bj0)--(ci0);
    \draw[->, >=stealth] (bj1)--(ci1);
    \draw[->, >=stealth] (bj0)--(ci1.south east);
\end{tikzpicture}\caption{Part of the poset $\Sigma_{\mathcal P}$ in \Cref{defin:S}, assuming, for $k\in \set{0,1}$, that  $i_k\in C\in \mathcal C$, $j_k\in D\in \mathcal C$, and $\set{i_k,j_k}, \set{i_k, j_{k-1}}\in \mathcal I$. Arrows denote only immediate strict inequalities; those arising from transitivity are not represented.}\label{figure:sigPcot}
\end{figure}
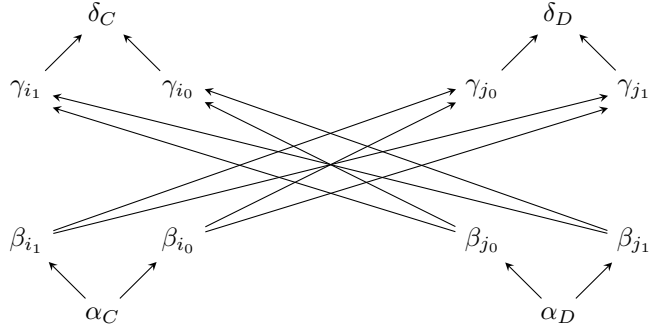
A part of the poset we just defined is pictured in \Cref{figure:sigPcot}.
\begin{pr}
If $\mathcal P$ is a maximal consistency pattern, then the structure $(\sigof{\mathcal P},<)$ is a strict partial order which satisfies the following properties.
\begin{enumerate}[label=(P\arabic*)]
\item\label{point:correctheights}  All $\alpha_C$ have height $0$, all $\beta_j$ have height $1$, all $\gamma_j$ have height $2$, and all $\delta_C$ have height $3$.
\item \label{point:Cu}   For all $C\ne C'\in \mathcal C$ we have $\alpha_C<\delta_{C'}$ and   $\alpha_C\centernot <\delta_C$.
\item \label{point:distance2}The following are equivalent:
\begin{enumerate*}[label=(\roman*)]
\item $j\in C$
\item $\alpha_C<\beta_j$ 
\item $\alpha_C\centernot< \gamma_j$
\item $\gamma_j<\delta_C$
\item $\beta_j\centernot <\delta_C$.
\end{enumerate*}
\end{enumerate}
\end{pr}
\begin{proof}
Let $R_0$ be the binary relation on $\Sigma_{\mathcal P}$ specified by points \ref{point:alb}, \ref{point:blc} and \ref{point:cld} of \Cref{defin:S}. Define $R_1\coloneqq R_0\cup (R_0\circ R_0)$. Notice that $(\alpha_C, \gamma_j)\in R_1$ if and only if there is $\set{i,j}\in \mathcal I$ such that $i\in \mathcal C$, if and only if $(\beta_j, \delta_C)\in R_1$. Hence,
by \ref{point:cmaximal}, we have  $R_1= R_0\cup\set{(\alpha_C, \gamma_j), (\beta_j, \delta_C)\mid j\notin C}$. Moreover, by \Cref{rem:weakmax}, we have that $R_2\coloneqq R_1\cup (R_1\circ R_1)$ equals $R_1\cup \set{(\alpha_C, \delta_{C'})\mid C\ne C'\in \mathcal C}$. Using that no $\delta_C$ ever appears as a first coordinate in the ordered pairs in $R_2$, it is easily checked that $R_2$ is transitive, hence coincides with the transitive closure of $R_0$, i.e.\ with $<$. This relation is transitive and irreflexive, that is, a strict partial order, and the three properties \ref{point:correctheights}, \ref{point:Cu}, \ref{point:distance2} are easily checked by using the description above.
\end{proof}
\begin{thm}
  Let $\mathsf{P}$ be a tree property defined via a maximal consistency pattern $\mathcal P$. Then $T$ has $\mathsf{P}$ if and only if $T$ has $\sopp{\mathcal P}$. 
\end{thm}
\begin{proof}
  Right to left, suppose that $\psi(x;y)$ has \sopp{\mathcal P}.  We show that $\phi(x;y_1,y_2)\coloneqq\psi(x;y_1)\land \neg \psi(x; y_2)$ has $\mathsf{P}$, witnessed by $b^j\coloneqq(d^{\beta_j}, d^{\gamma_j})$. Suppose that $\set{i,j}\in \mathcal I$. By definition, the pair $\set{\psi(x;\beta_j),  \psi(x; \gamma_i)}$ is inconsistent, so   $\models \psi(x;\beta_j)\implica \neg \psi(x; \gamma_i)$. Hence $\phi(x;\beta_i,\gamma_i)\land \phi(x;\beta_j,\gamma_j)$ is inconsistent as well. We conclude by showing that, for every $C\in \mathcal C$, the set $\set{\phi(x; b^j)\mid j\in C}$ is consistent.   For every $j\in C$ we have $\alpha_C< \beta_j$, hence by definition of $\sopp{\mathcal P}$ we have   $\models\psi(x; d^{\alpha_C})\implica \psi(x; d^{\beta_j})$. Therefore, it is enough to show consistency of the set of formulas $\psi(x; d^{\alpha_C})\cup \set{\neg\psi(x; d^{\gamma_j})\mid j\in C}$.

If this is not consistent, by compactness there are $k\in \omega\setminus \set 0$  and $u\in C^{[k]}$, such that $\models\psi(x;d^{\alpha_C})\implica\bigvee_{j\in u} \psi(x;d^{\gamma_j})$. But for $j\in u\subseteq C$ we have  $\models\psi(x;d^{\gamma_j})\implica \psi(x; d^{\delta_C})$, hence $\models\psi(x;d^{\alpha_C})\implica \psi(x;d^{\delta_C})$. This implies $\alpha_C<\delta_C$, which contradicts property~\ref{point:Cu}.

    Left to right, by \Cref{rem:negcons} there are witnesses $\phi(x;y)$ and $(b^j\mid j\in J)$   of $\mathsf{P}$ with the additional property that $\set{\neg\phi(x; b^j)\mid j\in J}$  has infinitely many realisations. Fix the following.
    \begin{itemize}
    \item  Pairwise distinct parameters $0,1,2,3$, of the same sort.
    \item For $j\in J$,  pairwise distinct realisations $e^j\models \set{\neg \phi(x; b^i)\mid i\in J}$.
    \item For every $C\in \mathcal C$,  some $a^C\models \set{\phi(x; b^j)\mid j\in C}$. 
    \end{itemize}
 Let $t$ and $z$ be tuples of variables of the same length as $0$ and $x$ respectively, and define the partitioned formula $\psi(x;y,z,t)$ as follows:
    \begin{align*}
      \psi(x;y,z,t)\coloneqq\phantom{.}& \Bigl((t=0)\land (x=z)\Bigr)
                 \lor\Bigl((t=1)\land (\phi(x; y)\lor x=z)\Bigr)\\
                 \phantom{.}\lor\phantom{.}&  \Bigl((t=2)\land (\neg \phi(x;y)\land x\ne z)\Bigr)
                 \lor  \Bigl((t=3)\land (x\ne z)\Bigl)
    \end{align*}
Now define the following parameters, where ``$\smiley$'' denotes an arbitrary (and unimportant) choice of a tuple of parameters in $\monster^{\abs y}$.
\begin{itemize}
\item For each $C\in \mathcal C$, set $d^{\alpha_C}\coloneqq (\smiley,a^C, 0)$ and $d^{\delta_C}\coloneqq(\smiley,a^C, 3)$.
\item For each $j\in J$, set $d^{\beta_j}\coloneqq(b^j,e^j, 1)$ and $d^{\gamma_j}\coloneqq(b^j, e^j, 2)$.
\end{itemize}
We prove that this yields the required embedding of posets.  Since the consistency pattern $\mathcal P$ is maximal, by \Cref{rem:weakmax} for $C\ne C'\in \mathcal C$ we have that $\set{\phi(x; b^j)\mid j\in C\cup C'}$ is inconsistent, hence the sets defined by the formulas  $\psi(x;d^{\alpha_C})$ are pairwise distinct  singletons. Similarly, the formulas  $\psi(x;d^{\delta_C})$ define pairwise distinct complements of singletons. Notice that no $\psi(x;d^-)$ is empty, nor the full $x=x$, hence these are indeed minimal and maximal elements of our poset. Observe straight away that $\models \psi(x;d^{\alpha_C})\implica\psi(x;d^{\delta_{C'}})$ if and only if $\alpha_C<\delta_{C'}$. It is furthermore clear that complements of singletons are not included in singletons, nor in different complements of singletons, and that different singletons are not included in each other, hence the map we defined, when restricted to the subposet given by levels $0$ and $3$ of $\sigof{\mathcal P}$, is indeed an embedding.

It is easy to see that all definable sets at levels $1$ and $2$ have at least two points, and that so do their complements. Hence,  $\psi(x;d^{\alpha_C})$ includes no other set defined above, and $\psi(x;d^{\delta_C})$ is included in no other set defined above.

Consider now the formulas $\psi(x; d^{\beta_j})$. The $e^j$ ensure that none of them implies any other; in particular, the map $\beta_j\mapsto\psi(x; d^{\beta_j})$ is injective. If $j\in C$, then   $\models\psi(x; d^{\alpha_j})\implica\psi(x; d^{\beta_j})$ by construction; conversely, the latter can only happen if $j\in C$ because of \ref{point:cmaximal}. Therefore, the subposet given by levels $0$ and $1$  is correctly embedded. Dually, for the subposet given by levels $2$ and $3$, since our construction is symmetric.

Observe that $\psi(x;d^{\gamma_j})$ is equivalent to $\neg \psi(x;d^{\beta_j})$.  By~\ref{point:pairsdecided},  $\models\phi(x;b^i)\implica \neg \phi(x;b^j)$ if and only if  $\set{i,j}\in \mathcal I$. It follows easily that $\models\psi(x; d^{\beta_j})\implica \psi(x; d^{\gamma_i})$ if and only if $\set{i,j}\in \mathcal I$, if and only if $\beta_j<\gamma_i$. To prove that the subposet given by levels $1$ and $2$ is also correctly embedded, we have to show that for no $i,j\in J$ we have $\models\psi(x; d^{\gamma_i})\implica \psi(x;d^{\beta_j})$. Assume this is the case for some $i, j\in J$.  By \ref{point:allareinconsistentwithsome} there is $k\in J$ such that $k\ne i$ and  $\set{j,k}\in \mathcal I$. But then by construction $\models \psi(x; d^{\beta_k})\implica \psi(x; d^{\gamma_j})$, hence $\models\psi(x; d^{\beta_k})\implica\psi(x; d^{\beta_i})$, which we already proved impossible.

Above, we showed that every pair of consecutive levels of $\sigof{\mathcal P}$ is correctly embedded.   Inclusion of definable sets is automatically transitive, so in order to finish we only need to show that we do not have any additional inclusions between different levels. Note that we already proved that no definable set coming from level $n$ of $\sopp{\mathcal P}$ is included in any definable set coming from level $m<n$, and that the inclusions between levels $0$ and $3$ are precisely the correct ones. If we prove that there are no unwanted inclusions between levels $0$ and $2$, then the symmetry of our construction will also tell us that there is no unwanted inclusion between levels $1$ and $3$. By property~\ref{point:distance2}, this amounts to proving $j\notin C\iff\models\neg\phi(a^C;b^j)\land (a^C\ne e^j)$. If $j\in C$, then by construction $\models\phi(a^C;b^j)$. If $j\notin C$, again by construction $\models \neg \phi(a^C;b^j)$; by  \ref{point:cmaximal} there is $i\in C$ with $\set{i,j}\in \mathcal I$, hence $\models a^C\ne e^j$ because $\models \phi(a^C;b^i)$ but $\models\neg \phi(e^j;b^i)$.
\end{proof}

\section{Future directions}\label{sec:future}
We conclude by listing some natural questions and problems.

Clearly, if $\age(\Sigma)=\age(\Sigma')$, then $\sop\Sigma$ and $\sop{\Sigma'}$ are equivalent. The converse is false: if $\Sigma\coloneqq\omega^{<\omega}$ with the usual ordering, then it can be shown that $\sop\Sigma$ is equivalent to the usual $\mathsf{SOP}$, but clearly $\omega^{<\omega}$ and $\omega$ have different age. Note also that, by passing to negations, posets which are the reverse of each other define the same property. 
\begin{question}Is there a nice characterisation of those pairs of posets $\Sigma,\Sigma'$ such that, at the level of theories, $\sop\Sigma$ is equivalent to $\sop{\Sigma'}$?
\end{question}

Having shown that several known dividing lines may be \emph{defined} uniformly, it is natural to ask whether they may be \emph{treated} uniformly. For instance, consider \emph{one-variable theorems}, that is, statements of the form: if there is a partitioned formula $\phi(x;y)$ with property $\mathsf{P}$, then there is one with $\abs x=1$. Such statements have been proven for $\mathsf{OP}, \mathsf{IP}, \mathsf{SOP}_1, \mathsf{TP}_2$, $\mathsf{ATP}$, by using different techniques. It would be nice to have uniform treatment.
\begin{question}
Is there a uniform proof of one-variable theorems for $\sop\Sigma$?
\end{question}

  It is easy to see that every $\sop\Sigma$ can be phrased as a \emph{weak $\Gamma$-property}  in the sense of \cite[Definition~5.17]{shelah_what_1999}. 
\begin{question}
Is every weak $\Gamma$-property equivalent to one of the form $\sop\Sigma$?
\end{question}
Part of the following question is a special case of the previous one. By \Cref{rem:sop1sop2} the analogous question for $\mathsf{SOP}_1$ has a positive answer.
\begin{question}
  Which of these properties can be defined via a maximal consistency pattern, or more generally as $\sop{\Sigma}$ for a suitable $\Sigma$?
  \begin{enumerate}
  \item $\mathsf{TP}$.
  \item $\mathsf{SOP}_n$ for $n\ge 4$.
  \item Not being rosy.
  \end{enumerate}
\end{question}
By \Cref{rem:otherproperties}, we already know that, for $k\ge 2$, the answer for $\mathsf{IP}_k$ is negative. This motivates our final question.
\begin{question}
Is it possible to phrase in similar terms $\mathsf{IP}_k$, for $k\ge 2$,  by replacing posets by some other kind of structures?
\end{question}

\section*{Acknowledgements}
We thank G.~Conant for suggesting to parameterise the original definition of $\mathsf{SUP}$, which led to \Cref{defin:sigmasop}, and A.~Papadopoulos and N.~Ramsey for pointing out to us the alternative definition of $\mathsf{SOP}_3$ in \cite[Claim~2.19]{shelah_toward_1996} and the definition of weak $\Gamma$-property in \cite{shelah_what_1999}.

R.~Mennuni was supported  by the Italian project  PRIN 2017: ``Mathematical Logic: models, sets, computability'' Prot.~2017NWTM8RPRIN and by the German Research Foundation (DFG) via HI 2004/1-1 (part of the French-German ANR-DFG project GeoMod) and under Germany’s Excellence Strategy EXC 2044-390685587, `Mathematics Münster: Dynamics-Geometry-Structure'.

\end{document}